\documentclass{article}
\usepackage{amssymb}
\usepackage{graphicx}
\usepackage{amsmath}
\usepackage{algorithmic}
\usepackage{caption}

\newtheorem{theorem}{Theorem}

\newtheorem{definition}[theorem]{Definition}
\newtheorem{example}[theorem]{Example}

\newtheorem{proposition}[theorem]{Proposition}
\newtheorem{remark}[theorem]{Remark}

\newenvironment{proof}[1][Proof]{\textbf{#1.} }{\ \rule{0.5em}{0.5em}}

\begin{document}
\title{Exponentially $E$-preinvex and $E$-invex functions in mathematical programming}
\date{}
\author{Najeeb Abdulaleem  \\
Department of Mathematics, Mahrah University, Al-Mahrah, Yemen\\
Faculty of Mathematics and Computer Science, University of \L \'{o}d\'{z},\\  Banacha 22, 90-238 \L \'{o}d\'{z}, Poland\\
Department of Mathematics, Hadhramout University\\ Al-Mahrah, Yemen \\
e-mail:  nabbas985@gmail.com}
\maketitle

\begin{abstract}
In this paper, we introduce a new concept of generalized convexity for $E$-differentiable vector optimization problems. Namely, the notion of exponentially $E$-invexity is defined. Further, some properties and results of exponentially $E$-invex functions are studied. The sufficient optimality conditions are derived under appropriate (generalized) exponentially $E$-invexity hypotheses. To exemplify the application of our proposed concept, we have included an appropriate example.


\end{abstract}

\section{Introduction}
Multiobjective optimization serves as a valuable mathematical framework to address real-world challenges involving conflicting objectives found in engineering, economics, and decision-making. However, many studies have traditionally assumed convexity in these problems (see, for example, \cite{Kanniappan1983}, \cite{Singh1987}, \cite{Weir1987}). To broaden the scope beyond convexity assumptions in theorems related to optimality conditions and duality, various concepts of generalized convexity have been introduced (see, for example, \cite{Alizadeh1}, \cite{Alizadeh2}, \cite{Najeeb2022VEinvex}, \cite{Najeeb2023univex}, \cite{Abdulaleemtyp}, \cite{Abdulaleem2019}, \cite{Avriel1972}, \cite{JeyakumarM},
and others). One particularly beneficial generalization is invexity, as introduced by Hanson \cite{Hanson}. This involves considering differentiable functions, denoted as $f: M\rightarrow R,$ where $M$ is a subset of $R^n.$ For these functions, Hanson proposes the existence of an n-dimensional vector function $\eta: M\times M \rightarrow R^n$ such that, for
all $x,u \in M,$ the inequality
\begin{equation*}
f(x) - f(u) \geq \nabla f(u) \eta(x, u)
\end{equation*}
holds. Ben Israel and Mond \cite{Israel}, Hanson and Mond \cite{HansonMond},  Craven and Glover \cite{CravenGlover},  along with numerous others, have explored various aspects, applications, and broader concepts related to these functions.

Youness \cite{youness1999} initially introduced the concept of $E$-convexity. Recently, there has been considerable interest in expanding the idea of $E$-convexity to novel classes of generalized $E$-convex functions, and researchers have investigated their characteristics (see, for example, \cite{Najeeb20191},    \cite{AntczakNajeeb2019s},  \cite{Abdulaleemtyp}, \cite{Abdulaleem2019}, \cite{Fulga2009}, \cite{Hu2007}, \cite{IqbalAhmadAli}, \cite{IqbalAliAhmad}, \cite{KiliSaleh2015},   \cite{Mirzapour}, \cite{Nanda}, \cite{PiaoJiao}, \cite{SyauLee},  \cite{SoleimaniDamaneh}, \cite{Yang}, \cite{AntczakNajeeb2020},  and others). Antczak \cite{Antczak2001} discussed the applications of exponentially convex functions in the mathematical programming and optimization theory. Following the research by Hanson and Craven, various forms of differentiable functions have emerged, aiming to extend the concept of invex functions. One such function involves exponential functions (see, for example, \cite{Alirezaei}, \cite{Avriel1972}, \cite{Noor}, \cite{NoorNoor2019}, \cite{NoorNoore2019}, \cite{NoorNoore2020}, \cite{Pecaric}, \cite{Pal2018}, and others).

In this paper, a new class of nonconvex $E$-differentiable vector optimization problems
with both inequality and equality constraints is considered in which the involved functions are exponentially (generalized) $E$-invex. Therefore, the concepts of exponentially pseudo-$E$-invex and exponentially quasi-$E$-invex functions for $E$-differentiable vector optimization
problems are introduced. Furthermore, we derive the sufficiency of the so-called $E$-Karush-Kuhn-Tucker optimality conditions for the considered $E$-differentiable vector optimization problem under appropriate exponentially (generalized) $E$-invexity hypotheses.
 This result is illustrated by suitable example of smooth multiobjective optimization problem in which the involved functions are exponentially (generalized) $E$-invex functions.

\section{Preliminaries}
Throughout this paper, the following conventions vectors $x=\left(
x_{1},x_{2},...,x_{n}\right) ^{T}$ and\ $y=\left(
y_{1},y_{2},...,y_{n}\right) ^{T}$ in $R^{n}$ will be followed:

(i) \ \ \ $x=y$ \ if and only if $x_{i}=y_{i}$ for all $i=1,2,...,n$;

(ii)\ \ \ $x>y$ \ if and only if $x_{i}>y_{i}$ for all $i=1,2,...,n$;

(iii) \ $x\geqq y$ \ if and only if $x_{i}\geqq y_{i}$ for all $i=1,2,...,n$;

(iv)\ \ $x\geq y$ \ if and only if $x_{i}\geqq y_{i}$ for all $i=1,2,...,n$
but $x\neq y$;

(v) \ \ $x\ngtr y$ \ is the negation of $x > y.$

\begin{definition}\cite{Najeeb20191}\label{def E-invex set}
Let $E:R^{n}\rightarrow R^{n}$. A set $M\subseteq R^{n}$ is said to be an $E$-invex set  if and only if there exists a vector-valued function $\eta: M\times M\rightarrow R^{n}$ such that the  relation%
\begin{equation*}
E\left(x_{0}\right) +\tau \eta \left( E\left( x\right),E\left(x_{0}\right) \right)
\in M
\end{equation*}
holds for all $x,x_{0}\in M$ and any $\tau \in \left[ 0,1\right] $.
\end{definition}

\begin{remark}\label{re1}
If $\eta$  is a vector-valued function defined by $\eta(z,y)=z-y$, then the definition of an $E$-invex set reduces to the definition of an $E$-convex set (see Youness \cite{youness1999}).
\end{remark}

\begin{definition}
 \label{def. E-convex function}Let $E:R^{n}\rightarrow R^{n}$. A function $f:M\rightarrow R$ is said to be $E$-preinvex  on $M$ if and only if the following inequality%
\begin{equation}
f\left( E\left(x_{0}\right) +\tau \eta \left( E\left( x\right),E\left(x_{0}\right) \right)
\right) \leqq \tau  f\left( E\left( x\right) \right) +\left( 1-\tau
\right) f\left( E\left(x_{0}\right) \right)  \label{2}
\end{equation}%
holds for all $x,x_{0}\in M$ and any $\tau  \in \left[ 0,1\right] $.
\end{definition}

Now, we introduce a new  concept of the exponentially $E$-preinvex  function.

\begin{definition}
 \label{def e-E-convex function}Let $E:R^{n}\rightarrow R^{n}.$ A  function $%
f:M\rightarrow R$ is said to be exponentially $E$-preinvex   on $M$ if and only if the following inequality%
\begin{equation}
e^{f\left( E\left(x_{0}\right) +\tau \eta \left( E\left( x\right),E\left(x_{0}\right) \right)
\right)} \leqq \tau  e^{ f\left( E\left( x\right) \right)} +\left( 1-\tau
\right) e^{f\left( E\left(x_{0}\right) \right)}  \label{2e}
\end{equation}%
holds for all $x,x_{0}\in M$ and any $\tau  \in \left[ 0,1\right] $.

In other words, \eqref{2e} is equivalent to the fact that the following inequalities

\begin{equation}
f\left( E\left(x_{0}\right) +\tau \eta \left( E\left( x\right),E\left(x_{0}\right) \right)
\right) \leqq \log \left[\tau  e^{ f\left( E\left( x\right) \right)} +\left( 1-\tau
\right) e^{f\left( E\left(x_{0}\right) \right)}\right]  \label{3e}
\end{equation}%
holds for all $x,x_{0}\in M$ and any $\tau  \in \left[ 0,1\right] $.
\end{definition}

\begin{definition}
\label{def strictly e-E-convex function}Let $E:R^{n}\rightarrow R^{n}.$  A  function $%
f:M\rightarrow R$ is said to be strictly exponentially $E$-preinvex  on $M$ if and only if the following
inequality%
\begin{equation}
e^{f\left( E\left(x_{0}\right) +\tau \eta \left( E\left( x\right),E\left(x_{0}\right) \right)
\right)} <\tau  e^{f\left( E\left( x\right) \right)} +\left( 1-\tau  \right)
e^{f\left( E\left(x_{0}\right) \right)}  \label{3}
\end{equation}%
holds for all $x,x_{0}\in M$, $E(x)\neq E(x_{0})$, and any $\tau  \in \left(
0,1\right) $.
\end{definition}

\begin{definition}
\label{def e-quasi-E-convex function}Let $E:R^{n}\rightarrow R^{n}.$  A function $%
f:M\rightarrow R$ is said to be exponentially quasi-$E$-preinvex  on $M$ if and only if the following
inequality%
\begin{equation}
e^{f\left( E\left(x_{0}\right) +\tau \eta \left( E\left( x\right),E\left(x_{0}\right) \right)
\right)} \leqq \;\max \;\{e^{f(E(x))},e^{f( E(x_{0}))}\}  \label{4}
\end{equation}%
holds for all $x,x_{0}\in M$ and any $\tau  \in \left[ 0,1\right] $.
\end{definition}

\begin{definition}
\label{def e-quasi-E-convex function}Let $E:R^{n}\rightarrow R^{n}.$  A function $%
f:M\rightarrow R$ is said to be strictly exponentially quasi-$E$-preinvex   on $M$ if and only if the following
inequality%
\begin{equation}
e^{f\left( E\left(x_{0}\right) +\tau \eta \left( E\left( x\right),E\left(x_{0}\right) \right)
\right)} < \;\max\; \{e^{f(E(x))},e^{f( E(x_{0}))}\}  \label{4}
\end{equation}%
holds for all $x,x_{0}\in M (x\neq x_{0}),$ and any $\tau  \in \left( 0,1\right) $.
\end{definition}

Note that every exponentially preinvex function is exponentially quasi-$E$-preinvex  and every exponentially  $E$-preinvex  function is exponentially  quasi-$E$-preinvex. However, the converse is not true.

\begin{definition}
Let $E:R^{n}\rightarrow R^{n}.$ We define the $E$-epigraph of an exponentially function $f:M\rightarrow R$ as follows
\begin{equation*}
\text{epi$_{E}$}(f)=\{(E(x),\Upsilon)\in M\times R: e^{f(E(x))}\leq \Upsilon\}.
\end{equation*}
\end{definition}

\begin{theorem}
Let $E:R^{n}\rightarrow R^{n},$ $M$ be a nonempty $E$-invex subset of $R^{n}$ and $f:M\rightarrow R$ be an exponentially function. Then $f$ is $E$-preinvex  if and only if its $E$-epigraph is an $E$-invex set.
\end{theorem}

\begin{proof}
Let $f$ be an exponentially $E$-preinvex  function. Then for any $(E(x),\Upsilon)$  and $(E(x_{0}),\Gamma)\in$ \text{epi$_{E}$}($f$), we have
$e^{f(E(x))}\leqq \Upsilon,$ $e^{f(E(x_{0}))}\leqq \Gamma$ . Also, for each $\tau \in [0,1]$, we have
\begin{equation}
\begin{split}
e^{f\left( E\left(x_{0}\right) +\tau \eta \left( E\left( x\right),E\left(x_{0}\right) \right)
\right)} &\leqq \tau  e^{ f\left( E\left( x\right) \right)} +\left( 1-\tau
\right) e^{f\left( E\left(x_{0}\right) \right)} \\
&\leqq \tau  \Upsilon +\left( 1-\tau
\right) \Gamma.
\end{split}
\end{equation}
Thus,
$(E\left(x_{0}\right) +\tau \eta \left( E\left( x\right),E\left(x_{0}\right) \right) , \tau \Upsilon + (1-\tau) \Gamma )  \in$ \text{epi$_{E}$}($f$).
Hence, \text{epi$_{E}$}($f$) is $E$-invex.

Conversely, let \text{epi$_{E}$}($f$)  be an $E$-invex set,  $(E(x),e^{f(E(x))}) \in$ \text{epi$_{E}$}($f$)   and $(E(x_{0}),e^{f(E(x_{0}))}) \in$ \text{epi$_{E}$}($f$). Then for each $E(x),E(x_{0})\in M$ and each $\lambda \in [0,1]$, we have
\begin{equation}
(E\left(x_{0}\right) +\tau \eta \left( E\left( x\right),E\left(x_{0}\right) \right), \tau  e^{f(E\left( x\right))} +(1-\tau ) e^{f(E\left(x_{0}\right))})\in \text{epi}_{E}(f)
\end{equation}
and thus,
\begin{equation}
e^{f\left( E\left(x_{0}\right) +\tau \eta \left( E\left( x\right),E\left(x_{0}\right) \right)
\right)} \leqq \tau  e^{ f\left( E\left( x\right) \right)} +\left( 1-\tau
\right) e^{f\left( E\left(x_{0}\right) \right)}.
\end{equation}
Hence,  $f$ is exponentially $E$-preinvex  on $M.$
\end{proof}

We now give the characterization of an exponentially quasi-$E$-preinvex function in terms of $E$-preinvexity
of its level sets.

\begin{theorem}
Let $E:R^{n}\rightarrow R^{n} $ and $M$ be a nonempty $E$-invex subset of $R^{n}.$  A function $f:M\rightarrow R$ is  exponentially quasi-$E$-preinvex function if and only if the  level sets $L_{E}(f,\Upsilon)$ are  $E$-invex for all $\Upsilon \in R.$
\end{theorem}

\begin{proof}
Let $f$ be an exponentially quasi-$E$-preinvex function, and for $\Upsilon \in R$, let $E(x), E(x_{0})\in L_{E}(f,\Upsilon)$. Then
$e^{f(E(x))}\leqq \Upsilon,$ $e^{f(E(x_{0}))}\leqq \Upsilon$. Since $f$ is an exponentially quasi-$E$-preinvex  function, for each $\tau \in [0,1]$, we have
\begin{equation*}
e^{f\left( E\left(x_{0}\right) +\tau \eta \left( E\left( x\right),E\left(x_{0}\right) \right)
\right)} \leqq \;\max \;\{e^{f(E(x))},e^{f( E(x_{0}))}\}  \leqq \Upsilon
\end{equation*}%
that is, $ E\left(x_{0}\right) +\tau \eta \left( E\left( x\right),E\left(x_{0}\right) \right)\in L_{E}(f,\Upsilon)$ for each $\tau \in [0,1].$ Hence, $L_{E}(f,\Upsilon)$ is $E$-invex.

Conversely, let $E(x),E(x_{0})\in M$ and $\overline{\Upsilon}=\max \;\{e^{f(E(x))},e^{f( E(x_{0}))}\}.$ Then $E(x),E(x_{0})\in L_{E}(f,\overline{\Upsilon})$, and by $E$-invexity of $L_{E}(f,\overline{\Upsilon})$, we have $E\left(x_{0}\right) +\tau \eta \left( E\left( x\right),E\left(x_{0}\right) \right) \in L_{E}(f,\overline{\Upsilon})$ for each $\tau \in [0,1].$ Thus for  each $\tau \in [0,1]$,
\begin{equation*}
e^{f\left( E\left(x_{0}\right) +\tau \eta \left( E\left( x\right),E\left(x_{0}\right) \right)
\right)} \leqq  \overline{\Upsilon} =\max \;\{e^{f(E(x))},e^{f( E(x_{0}))}\}.
\end{equation*}
This completes the proof.
\end{proof}

\begin{definition}
\cite{MegahedGommaYounessBanna} \label{def E-differentiability}Let $E:R^{n}\rightarrow R^{n}$ and $f:M\rightarrow R$
be a (not necessarily) differentiable function at a given point $x_{0}\in M$. It is
said that $f$ is an $E$-differentiable function at $x_{0}$ if and only if $%
f\circ E$ is a differentiable function at $x_{0}$ (in the usual sense) and,
moreover,
\begin{equation*}
\left( f\circ E\right) \left( x\right) =\left( f\circ E\right) \left(
x_{0}\right) +\nabla \left( f\circ E\right) \left( x_{0}\right) \left( x-x_{0}\right)
+\theta \left(x_{0},x-x_{0}\right) \left\Vert x-x_{0}\right\Vert ,  \label{8}
\end{equation*}%
where $\theta \left( x_{0},x-x_{0}\right) \rightarrow 0$ as $x\rightarrow x_{0}$.
\end{definition}

Now, we introduce a new concept of generalized exponentially convexity for  $E$-differentiable functions.
\begin{definition}
\label{d differentiable ex E-invexity}
Let $E:R^{n}\rightarrow R^{n},$ $M$ be a nonempty open $E$-invex subset of $R^{n}$ and
 $f:M\rightarrow R$ be an  $E$-differentiable function at  $x_{0}$ on $M$. It is said that $f$ is  exponentially $E$-invex function  at  $x_{0}$ on $M$ if, there exist $\eta:M\times M\rightarrow R^{n}$  such that, for all $x \in M$, the inequality
\begin{equation}
e^{f\left( E\left( x\right) \right)} -e^{f\left( E\left(x_{0}\right) \right)} \geqq
\nabla f\left( E\left(x_{0}\right) \right)e^{f\left( E\left(x_{0}\right) \right)} \eta\left( E\left( x\right),E\left(x_{0}\right) \right) \text{ \ }  \label{9ed}
\end{equation}%
holds. If inequality \eqref{9ed} holds for any $x_{0}$ on $M$, then $f$ is exponentially $E$-invex function on $M.$
\end{definition}

\begin{definition}
\label{d differentiable st ex E-invexity}
Let $E:R^{n}\rightarrow R^{n},$ $M$ be a nonempty open $E$-invex subset of $R^{n}$ and
 $f:M\rightarrow R$ be an  $E$-differentiable function at  $x_{0}$ on $M$. It is said that $f$ is  exponentially strictly $E$-invex function  at  $x_{0}$ on $M$ if, there exist $\eta:M\times M\rightarrow R^{n}$  such that, for all $x \in M$, $x\neq x_{0}$ the inequality
\begin{equation}
e^{f\left( E\left( x\right) \right)} -e^{f\left( E\left(x_{0}\right) \right)} >
\nabla f\left( E\left(x_{0}\right) \right)e^{f\left( E\left(x_{0}\right) \right)} \eta\left( E\left( x\right),E\left(x_{0}\right) \right) \text{ \ }  \label{9ed}
\end{equation}%
holds. If inequality \eqref{9ed} holds for any $x_{0}$ on $M,$ $x\neq x_{0}$, then $f$ is exponentially strictly $E$-invex function on $M.$
\end{definition}

Now, we give the necessary condition for an $E$-differentiable  exponentially $E$-invex function.

\begin{proposition}
Let $E:R^{n}\rightarrow R^{n},$    $f:M\rightarrow R$ be an $E$-differentiable exponentially $E$-invex function (an $E$-differentiable exponentially strictly $E$-invex function) on $M.$  Then, the following inequality%
\begin{equation}
(\nabla f\left( E\left( x\right) \right) e^{f\left( E\left( x\right) \right)}-\nabla f\left( E\left(x_{0}\right)
\right)e^{f\left( E\left(x_{0}\right) \right)}) \eta\left( E\left( x\right),E\left(x_{0}\right) \right) \geqq 0\text{ \ }%
\left( >\right)  \label{10}
\end{equation}%
holds for all $x, x_{0}\in  M$.
\end{proposition}

\begin{proof}
Since $f$ is an $E$-differentiable exponentially $E$-invex function, by Definition \ref{d differentiable ex E-invexity}, the following inequalities
\begin{equation}
e^{f\left( E\left( x\right) \right)} -e^{f\left( E\left( x_{0}\right) \right)} \geqq
\nabla f\left( E\left( x_{0}\right) \right)e^{f\left( E\left( x_{0}\right) \right)} \eta\left( E\left( x\right),E\left(x_{0}\right) \right) ,\text{ \ }\left( >\right)
\end{equation}
\begin{equation}
e^{f\left( E\left( x_{0}\right) \right)} -e^{f\left( E\left( x\right) \right)} \geqq
\nabla f\left( E\left( x\right) \right)e^{f\left( E\left( x\right) \right)} \eta\left( E\left( x\right),E\left(x_{0}\right) \right) \text{ \ }\left( >\right)
\end{equation}
hold for all $x, x_{0} \in M.$ By adding above inequalities, we get that the following inequality%
\begin{equation}
(\nabla f\left( E\left( x\right) \right)e^{f\left( E\left( x\right) \right)} -\nabla f\left( E\left( x_{0}\right)
\right)e^{f\left( E\left( x_{0}\right) \right)}) \eta\left( E\left( x\right), E\left( x_{0}\right) \right) \geqq 0\text{ \ }%
\left( >\right)
\end{equation}%
holds for all $x, x_{0}\in M$.
\end{proof}

Now, we introduce the definitions of  generalized exponentially $E$-invex
functions. Namely, the following result gives the characterization of an exponentially quasi-$E$-invex function in terms of its
gradient.

\begin{definition}
\label{def quasi E-convex function E-diff}Let $E:R^{n}\rightarrow R^{n}$ and $f:M\rightarrow R$ be an  $E$-differentiable function at $x_{0}\in M$. The function $f$ is said to be an exponentially quasi-$E$-invex
 at $x_{0}$ on $M$ if the relation
\begin{equation}
e^{f\left( E\left( x\right) \right)} \leqq e^{f\left( E\left(x_{0}\right) \right)} \Longrightarrow \nabla \left( f\circ E\right) \left(x_{0}\right) e^{f\left( E\left(x_{0}\right) \right)}
\eta\left( E\left( x\right),E\left(x_{0}\right) \right) \leqq 0  \label{13}
\end{equation}%
holds for each $x\in M$. If (\ref{13}) is satisfied for every $x_{0}\in M$, then the function $f$ is said to be an exponentially quasi-$E$-invex on $M$.
\end{definition}

\begin{definition}
\label{def pseudo e-E-convex function}Let $E:R^{n}\rightarrow R^{n}$ and $f:M\rightarrow R$ be an
$E$-differentiable function at $x_{0}\in M$. The function $f$ is said to be an exponentially pseudo-$E$-invex
 at $x_{0}$ on $M$ if the  relation%
\begin{equation}
e^{f\left( E\left( x\right) \right)} <e^{f\left( E\left(x_{0}\right) \right)} \Longrightarrow \nabla \left( f\circ E\right) \left(x_{0}\right) e^{f\left( E\left(x_{0}\right) \right)}
\eta\left( E\left( x\right),E\left(x_{0}\right) \right) <0  \label{11}
\end{equation}%
holds for each $x\in M$. If (\ref{11}) is satisfied for every $x_{0}\in M$, then the function $f$ is said to be an exponentially pseudo-$E$-invex on $M$.
\end{definition}

\begin{definition}
\label{def strictly pseudo E-convex function}Let $E:R^{n}\rightarrow R^{n}$ and $f:M\rightarrow R$ be an
$E$-differentiable function at $x_{0}\in M$. The function $f$ is said to be an exponentially strictly pseudo-$E$-invex  at $x_{0}$ on $M$ if the relation%
\begin{equation}
e^{f\left( E\left( x\right) \right)} \leqq e^{f\left( E\left(x_{0}\right) \right)}\Longrightarrow \nabla \left( f\circ E\right) \left(x_{0}\right) e^{f\left( E\left(x_{0}\right) \right)}
\eta\left( E\left( x\right),E\left(x_{0}\right) \right) <0  \label{12}
\end{equation}%
holds for each $x, x_{0}\in M$, $ x \neq x_{0}$. If (\ref%
{12}) is satisfied for every $x_{0}\in M$, $x\neq x_{0}$, then the function $f$ is said to be an exponentially
strictly pseudo-$E$-invex on $M$.
\end{definition}

Now, we present an example of such an exponentially pseudo-$E$-invex
function which is not exponentially $E$-invex.

\begin{example}
Let $E:R\rightarrow R$ and $f:R \rightarrow R$ be a nondifferentiable function at $x=-3$ defined by $%
f(x)=(x+3)^{\frac{1}{3}},$  $E(x)=(x+3)^{9}-3$ and $\eta(E(x),E(x_{0}))=x-x_{0}.$ The function $\left( f\circ E\right) (x)=(x+3)^{3}$ is a
differentiable function at $x=-3$, thus $f$ is an $E$-differentiable
function at $x=-3$.
Now we show that $f$ is exponentially pseudo-$E$-invex on $R$. Let $x,x_{0}\in R$ and $%
\tau \in \lbrack 0,1]$, and assume that $e^{\left( f\circ E\right) \left(
x\right)} <e^{\left( f\circ E\right) \left( x_{0}\right)}$. We have $e^{\left( f\circ
E\right) \left( x\right)} =e^{(x+3)^{3}}<e^{(x_{0}+3)^{3}}=e^{\left( f\circ E\right) \left(x_{0}\right)}$. This inequality implies that $x<x_{0}$. Hence, we have $$\nabla
\left( f\circ E\right) \left( x_{0}\right) e^{\left( f\circ E\right) \left( x_{0}\right)} \eta \left( E\left( x\right),E\left( x_{0}\right) \right) =3(x_{0}+3)^{2}e^{(x_{0}+3)^{3}}(x-x_{0})<0 .$$ Therefore, by
Definition \ref{def pseudo e-E-convex function}, $f$ is exponentially pseudo-$E$-invex on $R$.\newline
Further, it can be shown that $f$ is also exponentially quasi-$E$-invex on $R$. Assume
that $e^{\left( f\circ E\right) \left( x\right)} \leqq e^{\left( f\circ E\right)\left(x_{0}\right)}$. We have $e^{\left( f\circ E\right) \left( x\right)}
=e^{\left( f\circ E\right) \left(x_{0}\right)}=e^{(x+3)^{3}}\leqq e^{(x_{0}+3)^{3}}$. This
inequality implies that $x\leqq x_{0}$. Hence, we have
 $$\nabla \left( f\circ E\right) \left( x_{0}\right) e^{\left( f\circ E\right) \left( x_{0}\right)}\eta\left( E\left( x\right),E\left( x_{0}\right) \right)
=3(x_{0}+3)^{2}e^{(x_{0}+3)^{3}}(x-x_{0})\leqq 0.$$ Therefore, by Definition \ref{def
quasi E-convex function E-diff}, $f$ is exponentially quasi-$E$-invex on $R$.
\end{example}

\begin{definition}
\label{def E-minimizer} Let $E:R^{n}\rightarrow R^{n}$. It said that
 $\overline{x}\in R^{n}$ is a global $E$-minimizer of $%
f:M\rightarrow R$ if the inequality%
\begin{equation*}
f\left( E\left( \overline{x}\right) \right) \leqq f\left( E\left( x\right)
\right)
\end{equation*}%
holds for all $x\in M$.
\end{definition}

\begin{proposition}
Let $E:R^{n}\rightarrow R^{n}$  and $f:M\rightarrow R$
be an $E$-differentiable exponentially $E$-invex
function on $M.$   If $\nabla f\left(E(\overline{x})\right)=0$, then $\overline{x}$ is an $E$-minimizer of $f$.
\end{proposition}

\begin{proof}
Let $E:R^{n}\rightarrow R^{n}$. Further, assume that $%
f:M\rightarrow R$ is an $E$-differentiable exponentially $E$-invex
function on $M.$  Hence,
by Definition \ref{d differentiable ex E-invexity}, the inequality%
\begin{equation}
e^{f\left( E\left( x\right) \right)} -e^{f\left( E\left(\overline{x}\right) \right)} \geqq
\nabla f\left( E\left(\overline{x}\right) \right)e^{f\left( E\left(\overline{x}\right) \right)}
\eta\left( E\left( x\right),E\left(\overline{x}\right) \right)\label{18cc}
\end{equation}%
holds for all $x\in M.$  Since  $\nabla f\left(E(\overline{x})\right)=0$ and (\ref{18cc}), therefore, we have that the relation
\begin{equation}
e^{f(E(x))}-e^{f(E(\overline{x}))}\geqq 0
\end{equation}
implies that the inequality
\begin{equation*}
f\left( E\left(\overline{x}\right) \right) \leqq f\left( E\left( x\right)
\right)
\end{equation*}%
holds for all $x\in M$. This means, by Definition \ref{def E-minimizer}, that $\overline{x}$ is an $E$-minimizer of
$f$.
\end{proof}

\begin{proposition}
Let $E:R^{n}\rightarrow R^{n}$ be an operator and $f:M\rightarrow R$
be an $E$-differentiable exponentially pseudo-$E$-invex
function on $M.$   If $\nabla f\left(E(\overline{x})\right)=0$, then $\overline{x}$ is an $E$-minimizer of $f$.
\end{proposition}

\begin{proof}
The proof of this proposition follows from Definitions \ref{def pseudo e-E-convex function} and \ref{def E-minimizer}.
\end{proof}

\section{$E$-optimality conditions for multiobjective programming problems}
{\text{ } \\}

Consider the following multiobjective programming problem (VP):
\begin{equation*}
\begin{array}{c}
\text{minimize }f(x)=\left( f_{1}\left( x\right) ,...,f_{p}\left( x\right)
\right) \medskip \\
\text{subject to }g_{k}(x)\leqq 0\text{, \ }k\in K=\left\{ 1,...,m\right\}
,\medskip \\
\hspace{0.67in}h_{j}(x)=0\text{, \ }j\in J=\left\{ 1,...,q\right\} ,\medskip
\end{array}%
\qquad \text{(VP)}
\end{equation*}%
where  $f_{i}:R^{n}\rightarrow R$$(i\in I=\{1,2,...,p\})$, $g_{k}:R^{n}\rightarrow R$ $(k\in K)$ and $%
h_{j}:R^{n}\rightarrow R$ $(j\in J)$, are $E$-differentiable functions defined on $R^{n}.$
Let
\begin{equation*}
\Omega :=\left\{ x\in R^{n}:g_{k}(x)\leqq 0\text{, \ }k\in K\text{, }h_{j}(x)=0%
\text{, \ }j\in J\right\}
\end{equation*}%
be the set of all feasible solutions of (VP). Further, by $K\left( x\right),$ the set of inequality constraint indices that are active at a feasible
solution $x$, that is, $K\left( x\right) =\left\{ k\in K:g_{k}(x)=0\right\}.$

Let $E : R^{n}\rightarrow R^{n}$ be a given one-to-one and onto operator. Now, for the $E$-differentiable vector optimization problem (VP), we define its associated differentiable vector optimization problem (VP$_{E}$) as follows:
\begin{equation*}
\begin{array}{c}
\text{minimize }f(E(x))=\left( f_{1}\left( E(x)\right) ,...,f_{p}\left( E(x)\right)
\right) \medskip \\
\text{subject to }g_{k}(E(x))\leqq 0\text{, \ }k\in K=\left\{ 1,...,m\right\}
,\medskip \\
\hspace{0.67in}h_{j}(E(x))=0\text{, \ }j\in J=\left\{ 1,...,q\right\}.\medskip
\end{array}%
\qquad \text{(VP$_{E}$)}
\end{equation*}
Let
$
\Omega_{E} :=\left\{ x\in R^{n}:g_{k}(E(x))\leqq 0\text{, \ }k\in K\text{, }h_{j}(E(x))=0%
\text{, \ }j\in J\right\}
$
be the set of all feasible solutions of (VP$_{E}$).

\begin{definition}
\label{def weak Pareto point} A feasible point $E(\overline{y})$ is said to be
a weak $E$-Pareto (weakly $E$-efficient) solution for (VP) if and only if there
exists no feasible point $E(x)$ such that
\begin{equation*}
f(E(x))<f(E(\overline{y}))\text{.}
\end{equation*}
\end{definition}

\begin{definition}
\label{def Pareto point} A feasible point $E(\overline{y})$ is said to be an
$E$-Pareto ($E$-efficient) solution for (VP) if and only if there exists no feasible
point $E(x)$ such that
\begin{equation*}
f(E(x))\leq f(E(\overline{y}))\text{.}
\end{equation*}
\end{definition}

\begin{remark}
Let $E(\overline{y})\in \Omega$  be an $E$-Pareto solution (a weak $E$-Pareto solution)  of the
problem (VP). Then, $\overline{y}\in \Omega_{E}$ is a Pareto solution (a weak Pareto solution) of the problem (VP$_{E}$).
\end{remark}

\begin{theorem}\cite{AntczakNajeeb2020}
\label{KKT conditions for (VP)}($E$-Karush-Kuhn-Tucker necessary optimality
conditions). Let  $E\left( \overline{y}\right) $ be a weak $E$-Pareto
solution of  (VP). Moreover, let  $f_{i}$ ($i\in I$),  $g_{k}$ ($k\in K$), and
$h_{j}$ ($j\in J$), be $E$-differentiable  and the Kuhn-Tucker constraint qualification be satisfied at $\overline{y}$. Then
there exist  $\overline{\tau }\in R^{p}$, $\overline{\rho  }\in R^{m}$ and $%
\overline{\xi }\in R^{q}$ such that%
\begin{equation}
\begin{array}{c}
\sum_{i=1}^{p}\overline{\tau }_{i}\nabla
 f_{i}(E (\overline{y}))+
\sum_{k=1}^{m}\overline{\rho  }_{k}\nabla g_{k}(
E(\overline{y}))+\sum_{j=1}^{q}\overline{\xi }_{j}\nabla
h_{j}( E (\overline{y}))=0\text{,}%
\end{array}
\label{31}
\end{equation}%
\begin{equation}
\overline{\rho  }_{k} g_{k}( E(\overline{y}))=0\text{, \ }%
k\in K\text{,}  \label{32}
\end{equation}%
\begin{equation}
\overline{\tau }\geq 0\text{, \
}\overline{\rho  }\geqq 0\text{.}  \label{33}
\end{equation}
\end{theorem}

\begin{definition}
\label{def KKT point} $\left( E(\overline{y}),\overline{\tau },\overline{\rho
},\overline{\xi }\right) \in \Omega\times R^{p}\times R^{m}\times R^{q}$
is said to be an $E$-Karush-Kuhn-Tucker point ($E$-KKT point) for
(VP) if the relations (\ref%
{31})-(\ref{33}) are satisfied at $E(\overline{y})$ with Lagrange multipliers $%
\overline{\tau }$, $\overline{\rho  }$, $\overline{\xi }$.
\end{definition}

Now, we prove the sufficiency of the $E$-Karush-Kuhn-Tucker necessary
optimality conditions for the considered $E$-differentiable vector
optimization problem (VP) under exponentially $E$-invexity hypotheses.

\begin{theorem}
\label{theorem sufficiency weak Pareto under E-convexity}Let $\left(
E(\overline{y}),\overline{\tau },\overline{\rho  },\overline{\xi }\right) \in
\Omega \times R^{p}\times R^{m}\times R^{q}$ be an $E$-KKT point
of (VP). Let $J_{E}^{+}\left(
E\left( \overline{y}\right) \right) =\left\{ j\in J:\overline{\xi }%
_{j}>0\right\} $ and $J_{E}^{-}\left( E\left( \overline{y}\right) \right)
=\left\{ j\in J:\overline{\xi }_{j}<0\right\} $. Furthermore, assume the
following hypotheses are fulfilled:

\begin{enumerate}
\item[a)]  $f_{i}$, $i\in I$, is  $E$-differentiable exponentially $E$-invex function at $\overline{y}$ on $\Omega_{E}$,

\item[b)] $g_{k}$, $k\in K\left(E( \overline{y})\right) $, is  $E$-differentiable exponentially $E$-invex function at $\overline{y}$ on $\Omega_{E}$,

\item[c)] $h_{j}$, $j\in J^{+}\left( E\left(
\overline{y}\right) \right) $, is  $E$-differentiable exponentially $E$-invex function at $\overline{y}$ on $\Omega_{E}$,

\item[d)]  $-h_{j}$, $j\in J^{-}\left( E\left( \overline{y}%
\right) \right) $, is  $E$-differentiable exponentially $E$-invex function at $\overline{y}$ on $\Omega_{E}$.
\end{enumerate}

Then $E\left( \overline{y}\right) $ is a weak $E$-Pareto solution of (VP).
\end{theorem}

\begin{proof}
By assumption, $\left( E(\overline{y}),\overline{\tau },\overline{\rho  },%
\overline{\xi }\right) \in \Omega \times R^{p}\times R^{m}\times R^{q}$ is an
$E$-KKT point of (VP). Then, by Definition \ref{def KKT point}, the relations (\ref{31})-(\ref{33}) are satisfied at $%
E(\overline{y})$ with Lagrange multipliers $\overline{\tau }\in R^{p}$, $%
\overline{\rho  }\in R^{m}$ and $\overline{\xi }\in R^{q}$. We proceed by
contradiction. Assume, contrary to the conclusion, that  $E(\overline{y})$ is not a
weak $E$-Pareto solution of (VP). Hence, by Definition \ref%
{def weak Pareto point}, there exists another $E(x^{\star})\in \Omega $
such that%
\begin{equation}
f(E\left( x^{\star}\right) )<f\left( E\left( \overline{y}\right) \right)
\text{.}  \label{37}
\end{equation}%
Using hypotheses a)-d), by Definition \ref{def e-E-convex function} and Definition \ref{d differentiable ex E-invexity}, the following
inequalities%
\begin{equation}
e^{f_{i}\left( E\left( x^{\star}\right) \right)} -e^{f_{i}\left( E\left(
\overline{y}\right) \right)} \geqq \nabla f_{i}\left( E\left( \overline{y}\right) \right)e^{f_{i}\left( E\left(
\overline{y}\right) \right)} \eta\left( E\left( x^{\star}\right),E\left( \overline{y}\right) \right) \text{, }i\in I\text{,}  \label{38}
\end{equation}%
\begin{equation}
e^{g_{k}(E\left( x^{\star}\right) )}-e^{g_{k}(E\left( \overline{y}\right))}
\geqq \nabla g_{k}\left( E\left( \overline{y}\right) \right) e^{g_{k}(E\left( \overline{y}\right))}
 \eta\left( E\left( x^{\star}\right),E\left( \overline{y}\right) \right) \text{, }k\in K\left( E\left( \overline{y}\right) \right) \text{,}  \label{39}
\end{equation}
\begin{equation}
e^{h_{j}(E\left( x^{\star}\right) )}-e^{h_{j}(E\left( \overline{y}\right))}
\geqq \nabla h_{j}\left( E\left( \overline{y}\right) \right) e^{h_{j}(E\left( \overline{y}\right))}
 \eta\left( E\left( x^{\star}\right),E\left( \overline{y}\right) \right) \text{, }j\in J^{+}\left( E\left( \overline{y}\right) \right) \text{,}  \label{40}
\end{equation}%
\begin{equation}
-e^{h_{j}(E\left( x^{\star}\right) )}+e^{h_{j}(E\left( \overline{y}\right)
)}\geqq -\nabla h_{j}\left( E\left( \overline{y}\right) \right)e^{h_{j}(E\left( \overline{y}\right))}
 \eta\left( E\left( x^{\star}\right),E\left( \overline{y}\right) \right) \text{, \ }%
j\in J^{-}\left( E\left( \overline{y}\right) \right)  \label{41}
\end{equation}%
hold, respectively. Combining (\ref{37}) and (\ref{38}) and then multiplying
the resulting inequalities by the corresponding Lagrange multipliers and
adding both their sides, we get%
\begin{equation}
\left[ \sum_{i=1}^{p}\overline{\tau }_{i}\nabla \left( f_{i}\circ
E\right) \left( \overline{y}\right) \right]  \eta\left( E\left( x^{\star}\right),E\left( \overline{y}\right) \right) <0\text{.}  \label{42}
\end{equation}%
Multiplying inequalities (\ref{39})-(\ref{41}) by the corresponding Lagrange
multipliers, respectively, we obtain%
\begin{equation}
\overline{\rho  }_{k}e^{g_{k}(E\left( x^{\star}\right) )}-\overline{\rho  }%
_{k}e^{g_{k}(E\left( \overline{y}\right) )}\geqq \overline{\rho  }_{k}\nabla
g_{k}\left( E\left( \overline{y}\right) \right) e^{g_{k}(E\left( \overline{y}\right) )}
\eta\left( E\left( x^{\star}\right),E\left( \overline{y}\right) \right) \text{, }k\in K\left( E\left(
\overline{y}\right) \right) \text{,}  \label{43}
\end{equation}%
\begin{equation}
\overline{\xi }_{j}e^{h_{j}(E\left( x^{\star}\right) )}-\overline{\xi }%
_{j}e^{h_{j}(E\left( \overline{y}\right) )}\geqq \overline{\xi }_{j}\nabla
h_{j}\left( E\left( \overline{y}\right) \right) e^{h_{j}(E\left( \overline{y}\right) )}
\eta\left( E\left( x^{\star}\right),E\left( \overline{y}\right) \right) \text{, }j\in J^{+}\left(
E\left( \overline{y}\right) \right) \text{,}  \label{44}
\end{equation}%
\begin{equation}
\overline{\xi }_{j}e^{h_{j}(E\left( x^{\star}\right) )}-\overline{\xi }%
_{j}e^{h_{j}(E\left( \overline{y}\right) )}\geqq \overline{\xi }_{j}\nabla
h_{j}\left( E\left( \overline{y}\right) \right)e^{h_{j}(E\left( \overline{y}\right) )}
 \eta\left( E\left( x^{\star}\right),E\left( \overline{y}\right) \right) \text{, \ }j\in J^{-}\left(
E\left( \overline{y}\right) \right) \text{.}  \label{45}
\end{equation}%
Using the  condition (\ref{32})
together with $E(x^{\star})\in \Omega$ and $E(\overline{y})\in \Omega
 $, we get, respectively,%
\begin{equation}
\overline{\rho  }_{k}\nabla g_{k}\left( E\left( \overline{y}\right) \right)
\eta\left( E\left( x^{\star}\right),E\left( \overline{y}\right) \right)
\leqq 0\text{, }k\in K\left( E\left( \overline{y}\right) \right) \text{,}
\label{46}
\end{equation}%
\begin{equation}
\overline{\xi }_{j}\nabla h_{j}\left( E\left( \overline{y}\right) \right)
\eta\left( E\left( x^{\star}\right),E\left( \overline{y}\right) \right)
\leqq 0\text{, }j\in J^{+}\left( E\left( \overline{y}\right) \right) \text{,}
\label{47}
\end{equation}%
\begin{equation}
\overline{\xi }_{j}\nabla h_{j}\left( E\left( \overline{y}\right) \right)
\eta\left( E\left( x^{\star}\right),E\left( \overline{y}\right) \right)
\leqq 0\text{, \ }j\in J^{-}\left( E\left( \overline{y}\right) \right) \text{%
.}  \label{48}
\end{equation}%
Combining (\ref{42}) and (\ref{46})-(\ref{48}), we obtain that the following
inequality%
\begin{equation*}
\left[ \sum_{i=1}^{p}\overline{\tau }_{i}\nabla f_{i}(E\left( \overline{y}\right)) +
\sum_{k=1}^{m}\overline{\rho  }_{k}\nabla g_{k}\left( E\left( \overline{y}\right) \right) +\sum_{j=1}^{q}\overline{\xi
}_{j}\nabla h_{j}\left( E\left( \overline{y}\right) \right) \right] \eta\left( E\left( x^{\star}\right),E\left( \overline{y}\right) \right) <0
\end{equation*}%
holds, which is a contradiction to the  condition (\ref{31}).  Thus, the proof of this theorem is completed.
\end{proof}

If stronger $E$-differentiable exponentially $E$-invexity hypotheses are imposed on the functions constituting the considered vector optimization problems, then the sufficient optimality conditions for a feasible solution to be an  $E$-Pareto solution of the problem (VP) result is true.
\begin{theorem}
\label{theorem sufficiency Pareto under E-convexity}Let $\left( E(\overline{y}),%
\overline{\tau },\overline{\rho  },\overline{\xi }\right) \in \Omega \times
R^{p}\times R^{m}\times R^{q}$ be an $E$-KKT point of  (VP). Furthermore, assume that the
following hypotheses are fulfilled:

\begin{enumerate}
\item[a)]  $f_{i}$, $i\in I$, is  $E$-differentiable exponentially strictly $E$-invex function at $\overline{y}$ on $\Omega_{E}$,

\item[b)]  $g_{k}$, $k\in K\left( E(\overline{y})\right) $, is  $E$-differentiable exponentially $E$-invex function at $\overline{y}$ on $\Omega_{E}$,

\item[c)]  $h_{j}$, $j\in J^{+}\left( E\left(\overline{y}\right) \right) $, is  $E$-differentiable exponentially $E$-invex function at $\overline{y}$ on $\Omega_{E}$,

\item[d)] $-h_{j}$, $j\in J^{-}\left( E\left( \overline{y}%
\right) \right) $, is  $E$-differentiable exponentially $E$-invex function at $\overline{y}$ on $\Omega_{E}$.
\end{enumerate}
Then  $E\left( \overline{y}\right) $ is an $E$-Pareto solution of (VP).
\end{theorem}

\begin{proof}
The proof of this theorem is similar to the proof of Theorem \ref{theorem sufficiency weak Pareto under E-convexity} and is omitted.
\end{proof}

\begin{remark}
\label{remark sufficient quasi Econvex}According to the proof of Theorem \ref{theorem sufficiency weak Pareto under E-convexity}, the
sufficient conditions are also satisfied if  $%
g_{k}$, $k\in K_{E}\left( \overline{y}\right) $, $h_{j}$, $j\in J^{+}\left(
E\left( \overline{y}\right) \right) $, $-h_{j}$, $j\in J^{-}\left( E\left(
\overline{y}\right) \right) $, are  $E$-differentiable exponentially quasi-$E$-invex function at $\overline{y}$ on $\Omega_{E}$.
\end{remark}

Now, under the concepts of generalized $E$-differentiable exponentially $E$-invexity, we prove the sufficient optimality conditions for a feasible solution to be a weak $E$-Pareto solution of the problem (VP).

\begin{theorem}
\label{theorem sufficiency weak Pareto under generalized ex E-invexity} Let $%
\left( E(\overline{y}),\overline{\tau },\overline{\rho  },\overline{\xi }%
\right) \in \Omega \times R^{p}\times R^{m}\times R^{q}$ be an
$E$-KKT point of  (VP).
Furthermore, assume that the following hypotheses are fulfilled:

\begin{enumerate}
\item[a)]  $f_{i}$, $i\in I$, is $E$-differentiable exponentially pseudo-$E$-invex function at $\overline{y}$ on $\Omega_{E}$,

\item[b)]  $g_{k}$, $k\in K\left( E(\overline{y})%
\right) $, is  $E$-differentiable exponentially quasi-$E$-invex function at $\overline{y}$ on $\Omega_{E}$,

\item[c)]  $h_{j}$, $j\in J^{+}\left( E\left(
\overline{y}\right) \right) $, is  $E$-differentiable exponentially quasi-$E$-invex function at $\overline{y}$ on $\Omega_{E}$,

\item[d)] $-h_{j}$, $j\in J^{-}\left( E\left( \overline{y}%
\right) \right) $, is  $E$-differentiable exponentially quasi-$E$-invex function at $\overline{y}$ on $\Omega_{E}$.
\end{enumerate}
Then  $E\left( \overline{y}\right) $ is a weak $E$-Pareto solution of (VP).
\end{theorem}

\begin{proof}
By assumption, $\left( E(\overline{y}),\overline{\tau },\overline{\rho  },\overline{\xi }%
\right) \in \Omega \times R^{p}\times R^{m}\times R^{q}$ is
a Karush-Kuhn-Tucker point in the considered constrained $E$-optimization problem
(VP). Then, by
Definition \ref{def KKT point}, the Karush-Kuhn-Tucker necessary
optimality conditions (\ref{31})-(\ref{33}) are satisfied at $\overline{y}$
with Lagrange multipliers $\overline{\tau }\in R^{p}$, $\overline{\rho  }%
\in R^{m}$ and $\overline{\xi }\in R^{q}$. We proceed by contradiction.
Suppose, contrary to the result, that $E(\overline{y})$ is not a weak Pareto
solution in problem (VP). Hence, by Definition \ref{def Pareto point}, there exists another $E(x^{\curlywedge})\in \Omega$ such that%
\begin{equation}
f_{i}(E(x^{\curlywedge}))< f_{i}\left( E\left( \overline{y}\right) \right)
\text{, }i\in I\text{.}
\end{equation}%
Thus,
\begin{equation}
e^{f_{i}(E(x^{\curlywedge}))}< e^{f_{i}\left( E\left( \overline{y}\right) \right)}
\text{, }i\in I\text{.}  \label{26.a}
\end{equation}%
By hypothesis (a), the objective function $f$ is $E$-differentiable exponentially
 pseudo-$E$-invex  at $E(\overline{y})$ on $\Omega$. Then, \eqref{26.a} gives%
\begin{equation}
\nabla \left( f_{i}\circ E\right) \left( \overline{y}\right)  e^{f_{i}\left( E\left( \overline{y}\right) \right)}
\eta\left(E\left( x^{\curlywedge}\right), E\left( \overline{y}\right) \right) < 0\text{, }%
i\in I\text{,}  \label{28}
\end{equation}%
By the $E$-Karush-Kuhn-Tucker necessary optimality condition (\ref{33}) and $e^{f\left( E\left( \overline{y}\right) \right)}>0,$
inequality (\ref{28})  yields%
\begin{equation}
\left[ \sum_{i=1}^{p}\overline{\tau }_{i}\nabla \left( f_{i}\circ
E\right) \left( \overline{y}\right) \right] \eta\left(E\left( x^{\curlywedge}\right), E\left( \overline{y}\right) \right) <0\text{.}  \label{30a}
\end{equation}%
Since $E(x^{\curlywedge})\in \Omega $, $E(\overline{y})\in \Omega $ , therefore, the $E$-Karush-Kuhn-Tucker necessary optimality
conditions (\ref{32}) and (\ref{33}) imply%
\begin{equation*}
g_{k}(E\left( x^{\curlywedge}\right) )-g_{k}(E\left( \overline{y}\right)
)\leqq 0\text{, \ }k\in K\left( E\left( \overline{y}\right) \right) \text{.}
\end{equation*}%
Thus,
\begin{equation*}
e^{g_{k}(E\left( x^{\curlywedge}\right) )}\leqq e^{g_{k}(E\left( \overline{y}\right)
)}\text{, \ }k\in K\left( E\left( \overline{y}\right) \right) \text{.}
\end{equation*}%
From the assumption, each $g_{k}$, $k\in K$, is an $E$-differentiable exponentially quasi-$E$-invex function at $\overline{y}$ on $ \Omega_{E}$. Then, by Definition \ref{def quasi E-convex function E-diff}, we get%
\begin{equation}
\nabla g_{k}\left( E\left( \overline{y}\right) \right) e^{g_{k}\left( E\left( \overline{y}\right) \right)}
\eta\left(E\left( x^{\curlywedge}\right), E\left( \overline{y}\right) \right) \leqq 0\text{, \ }%
k\in K\left( E\left( \overline{y}\right) \right) \text{.}  \label{31a}
\end{equation}%
Thus, by the $E$-Karush-Kuhn-Tucker necessary optimality condition (\ref{33}), $e^{g\left( E\left( \overline{y}\right) \right)}>0,$
 and by Definition \ref{def quasi E-convex function E-diff}, (\ref{31a}) gives%
\begin{equation*}
\sum_{k\in K\left( E\left( \overline{y}\right) \right) }\overline{\rho  }%
_{k}\nabla g_{k}\left( E\left( \overline{y}\right) \right)
\eta\left(E\left( x^{\curlywedge}\right), E\left( \overline{y}\right) \right) \leqq 0\text{.}
\end{equation*}%
Hence, taking into account $\overline{\rho  }_{k}=0$, $k\notin K\left( E\left(
\overline{y}\right) \right) $, we have%
\begin{equation}
\sum_{k=1}^{m}\overline{\rho  }_{k}\nabla g_{k}\left( E\left( \overline{y}%
\right) \right)
\eta\left(E\left( x^{\curlywedge}\right), E\left( \overline{y}\right) \right)\leqq 0\text{.}  \label{32a}
\end{equation}%
From $E(x^{\curlywedge})\in \Omega $, $E(\overline{y})\in \Omega $, we get%
\begin{equation}
h_{j}(E\left( x^{\curlywedge}\right) )-h_{j}(E\left( \overline{y}\right) )=0%
\text{, }j\in J^{+}\left( E\left( \overline{y}\right) \right) \text{,}
\end{equation}%
\begin{equation}
-h_{j}(E\left( x^{\curlywedge}\right) )-\left( -h_{j}(E\left( \overline{y}%
\right) )\right) =0\text{, }j\in J^{-}\left( E\left( \overline{y}\right)
\right) \text{.}
\end{equation}%
Thus,
\begin{equation}
e^{h_{j}(E\left( x^{\curlywedge}\right) )}-e^{h_{j}(E\left( \overline{y}\right) )}=0%
\text{, }j\in J^{+}\left( E\left( \overline{y}\right) \right) \text{,}
\label{33a}
\end{equation}%
\begin{equation}
-e^{h_{j}(E\left( x^{\curlywedge}\right) )}-\left( -e^{h_{j}(E\left( \overline{y}%
\right) )}\right) =0\text{, }j\in J^{-}\left( E\left( \overline{y}\right)
\right) \text{.}  \label{34a}
\end{equation}%
Since each equality constraint $h_{j}$, $j\in J^{+}\left( E\left( \overline{x%
}\right) \right) $, and each function $-h_{j}$, $j\in J^{-}\left( E\left(
\overline{y}\right) \right) $, is an $E$-differentiable exponentially quasi-$E$-invex function at $\overline{y}$ on $\Omega_{E}$, then by Definition \ref{def quasi E-convex function E-diff}, $e^{h(E\left( \overline{y}\right))}>0,$ (\ref{33a}) and (\ref{34a})
give, respectively,%
\begin{equation}
\nabla h_{j}\left( E\left( \overline{y}\right) \right)
\eta\left(E\left( x^{\curlywedge}\right), E\left( \overline{y}\right) \right)\leqq 0\text{,\ }%
j\in J^{+}\left( E\left( \overline{y}\right) \right) \text{,}  \label{35a}
\end{equation}%
\begin{equation}
-\nabla h_{j}\left( E\left( \overline{y}\right) \right)
\eta\left(E\left( x^{\curlywedge}\right), E\left( \overline{y}\right) \right) \leqq 0\text{,\ }%
j\in J^{-}\left( E\left( \overline{y}\right) \right) \text{.}  \label{36a}
\end{equation}%
Thus, (\ref{35a}) and (\ref{36a}) yield%
\begin{equation*}
\bigg[ \sum_{j\in J^{+}\left( E\left( \overline{y}\right) \right) }\overline{%
\xi }_{j}\nabla h_{j}\left( E\left( \overline{y}\right) \right)
+\sum_{j\in
J^{-}\left( E\left( \overline{y}\right) \right) }\overline{\xi }_{t}\nabla
h_{j}\left( E\left( \overline{y}\right) \right) \bigg]
\eta\left(E\left( x^{\curlywedge}\right), E\left( \overline{y}\right) \right) \leqq 0\text{.}
\end{equation*}%
Hence, taking into account $\overline{\xi }_{j}=0$, $j\notin J^{+}\left(
E\left( \overline{y}\right) \right) \cup J^{-}\left( E\left( \overline{y}%
\right) \right) $, we have%
\begin{equation}
\sum_{j=1}^{q}\overline{\xi }_{j}\nabla h_{j}\left( E\left( \overline{y}%
\right) \right)
\eta\left(E\left( x^{\curlywedge}\right), E\left( \overline{y}\right) \right) \leqq 0\text{.}  \label{37a}
\end{equation}%
Combining (\ref{30a}), (\ref{32a}) and (\ref{37a}), we get that the following
inequality%
\begin{equation*}
\bigg[ \sum_{i=1}^{p}\overline{\tau}_{i}\nabla f_{i}(E\left( \overline{y}\right)) +\sum_{k=1}^{m}\overline{\rho  }_{k}\nabla g_{k}\left( E\left( \overline{%
x}\right) \right)
+\sum_{j=1}^{q}\overline{\xi }_{j}\nabla h_{j}\left(
E\left( \overline{y}\right) \right) \bigg]
\eta\left(E\left( x^{\curlywedge}\right), E\left( \overline{y}\right) \right) <0\text{}
\end{equation*}%
which is a contradiction to the $E$-Karush-Kuhn-Tucker necessary
optimality condition (\ref{31}). Thus, the proof of this theorem is
completed.
\end{proof}

In order to illustrate the above sufficient optimality conditions, we now present an example of an $E$-differentiable problem in which the involved functions are exponentially (generalized) $E$-invex.

\begin{example}
\label{example2}Consider the following nonconvex nondifferentiable vector
optimization problem
\begin{equation*}
\begin{array}{c}
\text{}f(x)=(\log(\sqrt[3]{x_{1}}+\sqrt[3]{x_{2}}+1)\text{ },
\text{ }\log(\sqrt[3]{x_{1}}+\sqrt[3]{x_{2}}+2))\;\; \rightarrow V-\min \medskip \\
\text{s.t. }g_{1}(x)=-\sqrt[3]{x_{1}}\leqq 0,\text{ \ }%
\medskip \\
g_{2}(x)=-\sqrt[3]{x_{2}}\leqq 0\text{.}%
\end{array}%
\text{ \ \ (VP1)}
\end{equation*}%
Note that $\Omega =\left\{ \left( x_{1},x_{2}\right) \in R^{2}:-\sqrt[3]{x_{1}}\leqq 0\text{
}\wedge \text{ }-\sqrt[3]{x_{2}}\leqq 0\right\}$. Let $E:R^{2}\rightarrow R^{2}$ be defined as
follows $E\left( x_{1},x_{2}\right) =\left( x_{1}^{3},x_{2}^{3}\right),$  $\eta(E(x),E(\overline{y}))=(1-e^{\overline{y}_{1}-x_{1}}, 1-e^{\overline{y}_{2}-x_{2}}).$ Now,
for the considered nonconvex nondifferentiable multiobjective programming
problem (VP1), we define its associated differentiable optimization problem (VP1$%
_{E}$) as follows%
\begin{equation*}
\begin{array}{c}
\text{}(f\circ E)(x)=(\log({x_{1}}+{x_{2}}+1)\text{ },
\text{ }\log({x_{1}}+{x_{2}}+2))\;\;\rightarrow V-\min\medskip \\
\text{s.t. }g_{1}(E(x))=-x_{1}\leqq 0,\text{ \ \ }\medskip \\
g_{2}(E(x))=-x_{2}\leqq  0\text{.}%
\end{array}%
\text{ \ \ (VP1}_{E}\text{)}
\end{equation*}%
Note that $\Omega _{E}=\left\{ \left( x_{1},x_{2}\right) \in
R^{2}:x_{1}\geqq 0\text{ }\wedge \text{ }x_{2}\geqq 0\right\} $ and $\overline{y}=\left( 0,0\right) $ is
the set of all feasible solutions of the problem (VP1$_{E}$). Further, note that all
functions constituting the considered vector optimization problem (VP1) are $%
E$-differentiable at $\overline{y}=\left( 0,0\right) $. Then, it can also be
shown that the $E$-Karush-Kuhn-Tucker necessary optimality conditions (\ref%
{31})-(\ref{33}) are fulfilled at $\overline{y}=\left( 0,0\right) $ with
Lagrange multipliers $\overline{\tau}_{1}=\frac{1}{2}$, $\overline{%
\tau }_{2}=\frac{1}{2}$ and $\overline{\rho}_{1}=\overline{\rho}_{2}=1$. Further, it can be proved that $f$ is an $E$-differentiable exponentially pseudo-$E$-invex
function at $\overline{y}$ on $\Omega _{E}$,  the constraint function $g_{1}$ and $g_{2}$ are an $E$-differentiable exponentially  quasi-$E$-invex function at $\overline{y}$ on $\Omega _{E}$.
Hence, by Theorem \ref{theorem sufficiency weak Pareto under generalized ex E-invexity}, $E(\overline{y})=\left(
0,0\right) $ is an $E$-Pareto solution of the optimization problem (VP).

\begin{figure}[ht]
  \begin{minipage}[b]{0.49\linewidth}
    \centering
    \includegraphics[width=\linewidth]{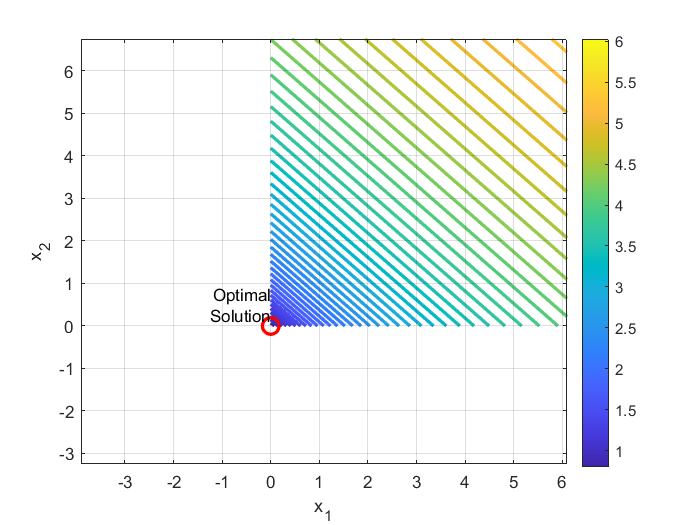}
  \end{minipage}
  \begin{minipage}[b]{0.5\linewidth}
    \centering
    \includegraphics[width=\linewidth]{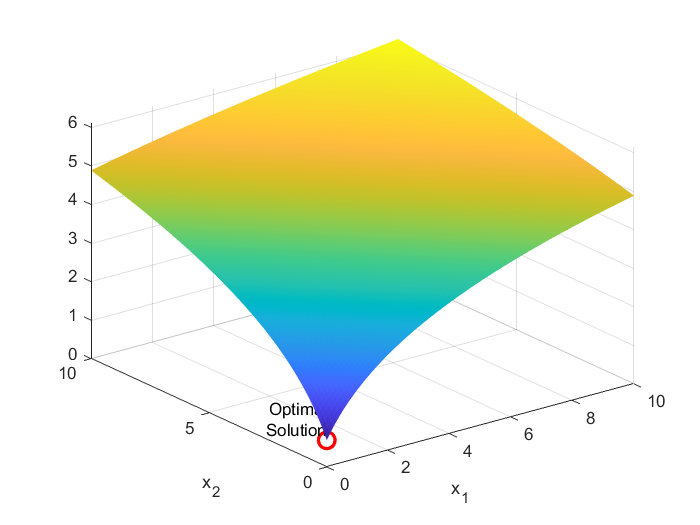}
  \end{minipage}
  \caption{Graphical view of (VP1$_{E}$).}
\end{figure}
\end{example}

\begin{remark}
Note that we are not able to use the optimality conditions for differentiable multiobjective programming problems in order to find efficient solutions in the vector optimization problem (VP1) considered in Example \ref{example2} since some of the involved functions  are not differentiable. Also, the sufficient optimality conditions with convexity hypotheses are not applicable for (VP1) since (VP1) is not a convex vector optimization problem.
\end{remark}

\section{Concluding remarks}
In this paper, a new class of nonconvex $E$-differentiable
vector optimization problems with both inequality and equality constraints have been considered. We have introduced the notion of exponentially $E$-invex functions, delved into their key characteristics, and expanded the framework by introducing various generalized exponentially $E$-invexity concepts.  Further, we have established sufficient optimality conditions for $E$-differentiable vector optimization problems under (generalized) exponentially $E$-invexity. To illustrate these findings, we have provided an example of nonconvex nonsmooth vector optimization problems.

However, some interesting topics for further research remain. It would be of interest to investigate whether it
is possible to prove similar results for other classes of $E$-differentiable vector optimization problems. We shall investigate these questions in subsequent papers.

\end{document}